\documentclass[review,onefignum,onetabnum]{siamart171218}

\usepackage{lipsum}
\usepackage{amsfonts}
\usepackage{graphicx}
\usepackage{epstopdf}
\usepackage{algorithmic}
\usepackage{longtable,tabularx}
\usepackage{amsmath,bm}
\usepackage{enumitem}
\usepackage[mathscr]{euscript}
\usepackage{wrapfig}

\ifpdf
  \DeclareGraphicsExtensions{.eps,.pdf,.png,.jpg}
\else
  \DeclareGraphicsExtensions{.eps}
\fi

\newsiamremark{remark}{Remark}
\newsiamremark{hypothesis}{Hypothesis}
\crefname{hypothesis}{Hypothesis}{Hypotheses}
\newsiamthm{claim}{Claim}

\headers{ }{Sameh A. Eisa and Sameer Pokhrel }

\title{Analyzing and Mimicking the Optimized Flight Physics of Soaring Birds: A Differential Geometric Control and Extremum Seeking System Approach with Real Time Implementation}

\author{Sameh A. Eisa \thanks{Aerospace Engineering and Engineering Mechanics Department at the University of Cincinnati, OH, USA
  (\email{eisash@ucmail.uc.edu, sameheisa235@hotmail.com}).}
\and Sameer Pokhrel \thanks{Aerospace Engineering and Engineering Mechanics Department at the University of Cincinnati, OH, USA
  (\email{pokhresr@mail.uc.edu}).}
}

\usepackage{amsopn}

\makeatletter
\newcommand*{\addFileDependency}[1]{
  \typeout{(#1)}
  \@addtofilelist{#1}
  \IfFileExists{#1}{}{\typeout{No file #1.}}
}
\makeatother

\usepackage{setspace}

\nolinenumbers
\ifpdf
\hypersetup{
  pdftitle={Analyzing and Mimicking the Optimized Flight Physics of Soaring Birds: A Differential Geometric Control and Extremum Seeking Approach with Real Time Implementation},
}
\fi

\begin{document}

\maketitle

\begin{abstract}
For centuries, soaring birds -- such as albatrosses and eagles -- have been mysterious and intriguing for biologists, physicists, aeronautical/control engineers, and applied mathematicians. These fascinating biological organisms have the ability to fly for long-duration while spending little to no energy, as they utilize wind for gaining lift. This flight technique/maneuver is called dynamic soaring (DS). For biologists and physicists, the DS phenomenon is nothing but a wonder of the very elegant ability of the bird's interaction with nature and using its physical ether in an optimal way for better survival and energy efficiency. For the engineering community, said DS phenomenon, is a source of inspiration and an unequivocal promising chance for bio-mimicking. In literature, significant work has been done on modeling and constructing control systems that allow the DS maneuver to be mimicked. Indeed, if one is able to construct a control system model that achieves what soaring birds actuate, then this developed framework, arguably, is able to capture many details about the mimicked biological behavior; hence, a bridge can be built between the biological and physical/engineering sides. However, mathematical characterization of the DS phenomenon in literature has been limited to optimal control configurations that utilized developments in numerical optimization algorithms along with control methods to identify the optimal DS trajectory taken (or to be taken) by the bird/mimicking system. Unfortunately, all said methods are highly complex and non-real-time. Hence, the mathematical characterization of the DS problem, we believe, appears to be at odds with the phenomenon/birds-behavior. In this paper, we provide a novel two-layered mathematical approach to characterize, model, mimic, and control DS in a simple and with \textit{real-time} implementation, which we believe, decodes better the biological behavior of soaring birds. The first layer will be a differential geometric control formulation and analysis of the DS problem, which allow us to introduce a control system that is simple, yet controllable. The second layer will be a linkage between the DS philosophy and a class of dynamical control systems known as extremum seeking systems. This linkage provides the control input that makes the DS a real-time reality. We believe our framework captures more of the biological behavior of soaring birds and opens the door for geometric control theory and extremum seeking systems to be utilized in systems biology and natural phenomena. Simulation results are provided along with comparisons with powerful optimal control solvers to illustrate the advantages of the introduced method.               
\end{abstract}



\section{Introduction}\label{sec:intro}
\subsection{Optimized flight physics of soaring birds}\label{sec:soaring_birds}
Soaring birds are many kinds and species \cite{ehrlich1994birdwatcher}, such as but not limited to albatrosses and eagles. The albatross flight physics has captured the interests of biologists, physicists, and engineers for its magnificent energy-efficient ability. Said flight maneuver, known as dynamic soaring (DS) \cite{langelaan2009enabling,mir2018optimal,MITbousquet2017optimal,richardson2011albatrosses,mir2021stability} enables soaring birds -- such as albatrosses -- to fly almost for free and with barely feeling the need to flap their wings due to their ability to acquire lift from the wind itself -- see the early Nature papers on the DS phenomenon \cite{rayleigh1883soaring,wilson1975sweeping}. The albatross optimized flight has been perceived as exceptionally interesting and at least fascinating and mysterious; this can be noticed in observations that go back to Lord Rayleigh \cite{rayleigh1883soaring} and Leonardo da Vinci \cite{richardson2019leonardo}. Refer \Cref{fig:dynamicsoaring} for a graphical depiction of the DS maneuver, which is a cycle typically consisting of four phases \cite{mir2018review}: (a) wind-ward climb, (b) high altitude turn, (c) leeward descend, and (d) low altitude turn. The DS cycle can be considered, ideally, as an energy neutral (i.e., constant energy level) or near neutral (i.e., near constant energy level). That is, thanks to the DS maneuver mentioned above and in \Cref{fig:dynamicsoaring}, the energy spent by the bird in the short duration of the execution of the DS cycle is substantially smaller when compared to the total energy it possesses at the beginning of the cycle; hence, the energy level is typically assumed to be a constant (or near constant) during the DS cycle -- refer \cite{mir2018review} for a detailed review on DS and \cite{mohamed2022opportunistic} for review on soaring birds. In order to perform DS, soaring birds fly into the headwind and gain lift due to the physical property known as ``wind shear." Wind shear occurs in areas where a significant change of wind speed with altitude takes place, which often happens above sea/oceans \cite{MITbousquet2017optimal,mir2018review}. The gain of lift allows the bird to gain height and climb in altitude up to a point. Then, the bird performs its high altitude turn when it can no longer sustain height gain and starts descending -- trading gained potential energy with kinetic energy. At the lowest altitude possible, the bird performs its low altitude turn and starts another DS cycle. By the end of every DS cycle, the bird ends up traveling a distance for free -- or at least partially for free, achieving energy neutrality or near neutrality as discussed earlier. Of note, the DS biological phenomenon has been validated/observed experimentally \cite{sachs2013experimental,yonehara2016flight}.

\begin{figure}[ht]
    \centering
\includegraphics[width=1\textwidth]{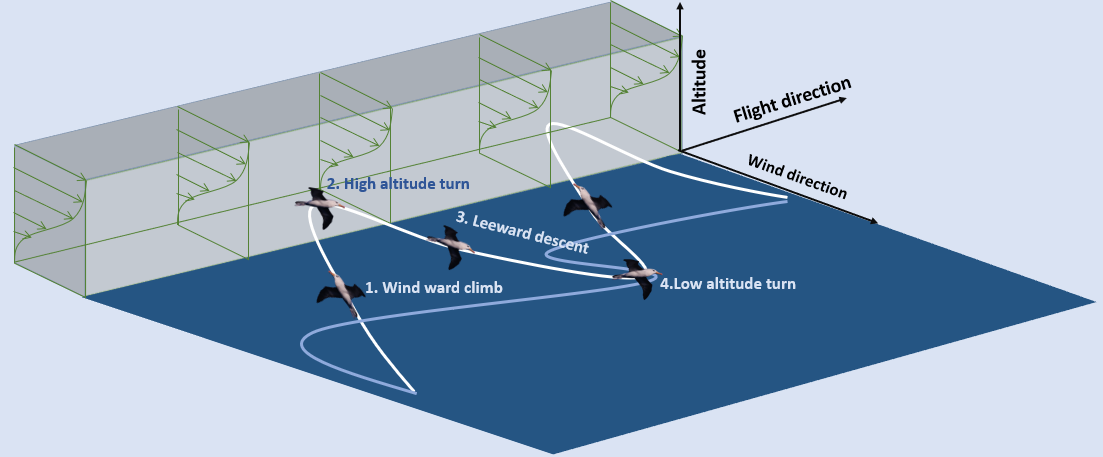}
\caption{DS maneuver in presence of wind shear}
\label{fig:dynamicsoaring}
\end{figure}

From a physical point of view \cite{denny2008dynamic,wilson1975sweeping}, the energy neutrality (or near neutrality) implies that the DS phenomenon is a conservative-like flying technique, which is extremely rare -- if not unique. That is, the energy harvested from the wind counters the energy -- traditionally lost in flight dynamic systems due to the non-conservative drag force. One can see that the DS phenomenon can appear to be more like an optimal control and/or dynamic optimization problem from an applied mathematics point of view \cite{cone1964mathematical,flanzer2012robust,sachs2005application}. As a matter of fact, this has been the trend of studying this problem for decades as documented in the review paper \cite[section 6]{mir2018optimal}. Many numerical optimization methods and techniques (see \cite{betts2010practical} for example) have been used for solving DS as a problem of finding the optimal trajectory -- that if taken by the bird/mimicking system, an energy neutral (or near neutral) DS cycle is achieved.    

\subsection{Motivation: dynamic soaring as a control-affine extremum seeking system, formulated and analyzed via geometric control theory}\label{sec:motivation} Realizing that this work connects and puts together multiple fields of studies, we believe the reader may benefit from explaining our motivation in detail. We could not find a better statement to represent the bigger picture of our motivation other than that of Brown, RHJ in his flight of birds study \cite{brown1963flight}: ``Bird flight can be studied neither as a problem in physics nor from the standpoint of biology alone. Both points of view are necessary and complementary." The fact that the DS problem has always been studied as a complex, computationally expensive, non-real-time optimal control problem \cite{mir2018review}, makes us believe that the decades-long literature of this problem is at odds with the biological behavior of albatross or other soaring birds. Indeed, soaring birds do not need complex computational power or a long time before assessing their action (control) as observed in nature \cite{sachs2013experimental}. Thus, we turned our attention to finding a mathematical control theory that fits more -- and can be natural to -- the DS phenomenon itself as done by soaring birds. In order to close and narrow down the gap between the physical/engineering side and the biological side, we need to find the right mathematical theory and tools that achieve two targets. First, we need a simple mathematical formulation for the DS problem, which assumes simple control inputs like those actuated by the bird (i.e., pitching and rolling \cite{mir2018review,MITbousquet2017optimal}); however, this simple formulation needs to be proven controllable (i.e., with the given simple control inputs, the system representing the bird/mimicking system should be able to reach anywhere in the state space). Second, we need to find a logical control law that actuates the DS system in real-time, hypothetically, similar to the bird. For the first target, we asses that geometric control theory (GCT) is the right mathematical control theory candidate. GCT encompasses elegant tools and methods from differential geometry, which enable controllability analysis and motion planning of underactuated (i.e., the number of control inputs is less than the dimension of the state space), nonlinear systems \cite{bullo2019geometric,sussmann1972controllability,hermann1977nonlinear,sussmann1987general,stefani1985local,lafferriere1991motion}; remarkably, this is exactly what DS is: a highly nonlinear, underactuated system \cite{mir2018review}. Hence, we are motivated to utilize GCT especially given its revived use in many recent useful theoretical and application advancements, where crucial information/conclusions need analysis beyond linearization or first-order approximation; see, for example, GCT applications in averaging analysis \cite{Maggia2020higherOrderAvg}, flight dynamics and control \cite{hassan2017geometric, hassan2019differential}, and network systems \cite{menara2020conditions}, among others.  
For the second target, we make a linkage between DS and a powerful class of control systems called extremum seeking control (ESC) systems \cite{TanHistory2010}. Said systems (ESCs) operate by steering a given dynamical system to the minimum/maximum (extremum) of an objective function, autonomously, and in real-time \cite{ariyur2003real,scheinker2017model,KRSTICMain}. Moreover, ESC systems are possibly operable with no model at all, and they do not require access to the mathematical expression of the objective function -- provided that the objective function is accessible through measurements. In other words, ESC systems are model-free control system technique that performs real-time optimization. Now, it can be seen why ESC systems can be thought of as a natural candidate to describe/decode optimized biological behaviors that are performed in real-time by organisms in nature. In fact, it is reasonable to assess that albatrosses (or soaring birds) conduct DS without analyzing the mathematical expressions of the wind profile, aerodynamics, or a given objective function. Moreover, they conduct the flight of DS in real-time \cite{sachs2013experimental} while sensing (measuring) their environment; for instance, in \cite{pennycuick2008information_nostrils,brooke2002gusts_nostrils} it is observed that the nostrils of the albatross act as if they are airspeed sensor. Hence, we believe the linkage here between ESC systems and DS is reasonable by itself, but also has two other dimensions which connect our two targets comprehensively: (i) ESC theory has adopted systems that are control-affine, which operate in a way that is more like motion planning methods of GCT, see for example \cite{ESCTracking,DURR2013}; and (ii) ESC systems, in general, have been strongly successful in characterizing and decoding many biological phenomena in literature such as fish locomotion \cite{fish1}, and chemotaxis-like motion in bacteria in \cite{liu2010stochastic}.

\subsection{Contribution and organization of this paper}\label{sec:contribution}
In this paper, we provide a comprehensive, novel, and real-time implementable framework that provides a control theory and structure for the DS phenomenon. We reason and make the case that the introduced framework is capturing/reflecting more of what soaring birds do in nature -- enabling closer steps towards practical and industrial possibilities of bio-mimicry of soaring birds by unmanned systems and robotics. Realizing that this paper incorporates concepts and studies in flight physics, geometric control theory, and extremum seeking systems, it is important to provide background from these fields before introducing the direct contributions of the paper. Section \ref{sec:two_prob_form} provides a brief overview of the nonlinear, non-real-time optimal control formulation typically adopted in the decades-long literature of DS \cite{mir2018review}. We do this for two main reasons: (i) the reader of this paper may be interested to learn more about the decades-long literature of studying DS to help them in a way to also see the novelty of this work; and (ii) later in section \ref{sec:simulations} we compare our newly introduced approach with an optimal control solver representing the literature; hence, the reader can follow how the optimal control solver actually operates inline with section \ref{sec:two_prob_form}. Section \ref{sec:three_gct} provides a brief, yet self-explanatory, background on nonlinear controllability analysis using GCT followed by our proposed differential geometric control formulation of the DS problem. The proposed DS system (in control-affine form) is very simple, and yet, we rigorously prove its controllability. In section \ref{sec:four_esc}, we propose a particular control law for the formulated control-affine DS system, making it effectively a control-affine ESC system. Hence, similar to ESC systems, our very simple control structure is operable even on a model-free basis with only access to measurements of the objective function (aiming at maximizing the energy gain); this is no stranger and so natural to soaring birds (recall for example that the nostrils of albatrosses act as if they are airspeed sensor \cite{pennycuick2008information_nostrils,brooke2002gusts_nostrils}). In that same section, we provide simulation results that show strong comparability between our real-time control-affine ESC implementation and non-real-time implementation done via GPOPS2 \cite{gpops2}, a highly capable optimal control solver which uses hp-adaptive Gaussian quadrature collocation methods and sparse nonlinear programming. In addition to the direct contributions that are related to DS and its GCT formulation/analysis and ESC characterization, this paper serves as yet another contribution in a narrow stream of contributions aiming at discovering/revealing ESC manifestations in systems biology and nature. Hence, we dedicate section \ref{sec:discussion} to recall and discuss some other works that have used ESC methods to study biological or natural phenomena.

\section{Problem Formulation of Dynamic Soaring as in literature}\label{sec:two_prob_form}
In this section, for the convenience of the reader who may not be familiar with this problem, we briefly summarize the problem formulation, models, and bounds/constraints necessary for the implementation of DS as generally done in the decades-long literature of the problem \cite{mir2018review}: nonlinear, non-real-time, model-dependent, heavily constrained optimal control problem. Recall from earlier discussions in sections \ref{sec:motivation} and \ref{sec:contribution}, this is how, for example, the numerical optimizer we use for comparison in section \ref{sec:simulations}. operates, \textit{however, this section is not needed for the novel control-affine ESC implementation} which may operate in a model-free basis without any bounds/constraints. 

\subsection{Modeling of Wind shear} Wind shear is the difference in wind velocity over a short region in the atmosphere. Wind shear is essential for energy harvesting via DS \cite{mir2018review,MITbousquet2017optimal,wilson1975sweeping}. Hence, it is crucial to have an accurate wind shear model for numerical-based optimization and optimal control methods. For this study, we consider vertical wind shear in which the wind velocity changes with the altitude. The wind velocity at sea level is virtually zero and it increases with the increase in altitude. To model this variation, many models have been used in the literature on optimal control of DS \cite{mir2018review}. Here, we consider a logistic function similar to \cite{MITbousquet2017optimal} and provided by \Cref{eqn:wind}:
\begin{equation}\label{eqn:wind}
    W(z)=\frac{W_0}{1+e^{-z/\delta}}.
\end{equation}
This model is parameterized by free stream wind speed $W_0$, and shear layer thickness $\delta$ as depicted in \Cref{fig:sigmoidWind} (left).
\begin{figure}[ht]
    \centering
    \includegraphics[width=1\textwidth]{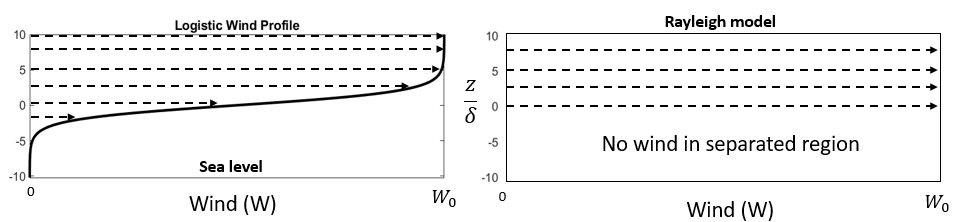}
    \caption{An example of a logistic wind profile with sea level at $z=-10$ (left) and Rayleigh discontinuous wind model (right).}
    \label{fig:sigmoidWind}
\end{figure}
It captures the features of any flow with a typical wind inhomogeneity $W_0$ developing over a length $\delta$ \cite{MITbousquet2017optimal}. For example, if the thickness of shear layer $\delta$ is very thin, this model approaches Rayleigh's discontinuous wind model as shown in \Cref{fig:sigmoidWind} (right).

\subsection{Flight dynamics model}
In the model-dependent optimal control literature of DS, three or six degrees of freedom models can be used. However, as noted in \cite{mir2018review}, six degrees of freedom models (as in \cite{6DOFakhtar2012positioning}) do not necessarily increase the accuracy as much, but increase the computational complexity/expense. Thus, we consider here a three-dimensional point-mass model with no thrust component to represent an albatross or a mimicking system model as follows \cite{MITbousquet2017optimal}:
\begin{equation}\label{eqn:dynamics}
    \begin{split}
        \dot{x}&=V \cos\gamma \cos\psi,\\
        \dot{y}&=V \cos\gamma \sin\psi-W,\\
        \dot{z}&= V \sin\gamma,\\
        m \dot{V}&= -D -mg\sin\gamma+m\dot{W} \cos\gamma \sin\psi,\\
        mV\dot{\gamma} &= L \cos \phi - mg \cos \gamma - m \dot{W} \sin \gamma \sin \psi, \\
        mV \dot{\psi }\cos \gamma  &= L \sin \phi + m \dot{W} \cos \psi. \\
    \end{split}
\end{equation}
Here, similar to \cite{MITbousquet2017optimal}, the orientation follows East, North and Up frame of reference, i.e. $(i,j,k)=(e_{East}, e_{North}, e_{Up})$ as shown in \Cref{fig:Angles}. Here, $\phi$ is the roll angle, $\psi$ is the heading angle, measured between $i$ and projection of velocity ($V$) in the $ij$-plane, and $\gamma$ is the flight path angle, measured between $V$ and $ij$-plane with nose up as positive. We assume the wind has density ($\rho$), velocity ($W$), and shear gradient ($\dot{W}$). At all times, the wind is assumed to blow from North to South resulting in the existence of only horizontal wind components. Moreover, the angles and speed are represented in the relative reference frame of the wind, whereas the position $(x,y,z)$ is represented in a fixed frame of the Earth. Lift (L) and drag (D) are represented as follows.
\begin{equation}\label{eqn:DS_general_system}
    \begin{split}
        L&= \frac{1}{2}\rho V^2 S C_L,\\
        D&= \frac{1}{2}\rho V^2 S C_D,
    \end{split}
\end{equation}
where the parabolic drag coefficient is given by $C_D = C_{D0} + KC_L^2$; $C_L$ is the lift coefficient, $C_D$ is the drag coefficient, $C_{D0}$ is zero-lift drag coefficient, and $K$ is the coefficient of induced drag. \Cref{tab:params} provides system's parameters as given in \cite{MITbousquet2017optimal} and references therein.
 \begin{table}[h]
\centering
\resizebox{\textwidth}{!}{
    \begin{tabular}{|>{\arraybackslash}p{1.8cm}|>{\arraybackslash}p{1cm}|>{\arraybackslash}p{1.5cm}|>{\arraybackslash}p{1cm}|>{\arraybackslash}p{1cm}|>{\arraybackslash}p{1cm}|>{\arraybackslash}p{1cm}|>{\arraybackslash}p{1cm}|>{\arraybackslash}p{0.5cm}|}
\hline
Parameter & mass (m) &wing area (S) & $C_{D0}$& K &  $\rho$ &  g & $W_0$ & $\delta$  \\
\hline
 Value & 9.5 kg & 0.65 $m^2$  &0.01   &0.0156   & 1.2 $kg/m^3$  &9.8 $m/s^2$ &7.8 m/s  & 7m\\
\hline
\end{tabular}
}
\caption{Characteristics of the albatross used in this study.}
\label{tab:params}
\end{table}


\begin{figure}
    \centering
    \includegraphics[width=0.8\textwidth]{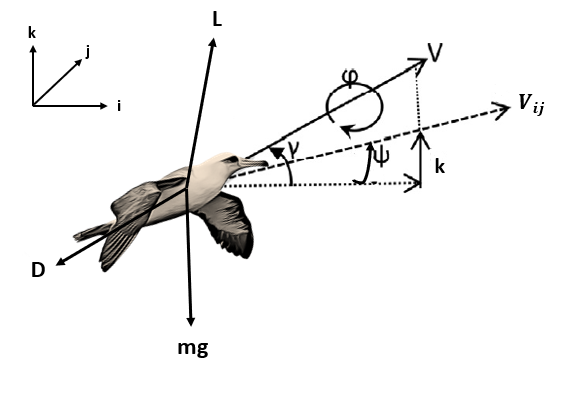}
    \caption{Aerodynamic forces acting on an albatross and the aerodynamic angles.}
    \label{fig:Angles}
\end{figure}

\subsection{Optimal control problem formulation}

Let the state vector be $\bm{x}(t)$ and the control vector be $\bm{u}(t)$ such that:
\begin{equation}\label{eqn:probForm}
\begin{split}
    \bm{x}(t)&=[x,y,z,V,\gamma,\psi],\\
    \bm{u}(t)&=[C_L,\phi].
\end{split}
\end{equation}
The flight dynamic system in \Cref{eqn:dynamics} with the wind model in \Cref{eqn:wind} can be written in the nonlinear system form:
\begin{equation}
    \bm{\dot{x}}(t)=\bm{f}(\bm{x}(t),\bm{u}(t)).
\end{equation}
The objective function we aim at maximizing (called performance index at times) - is denoted as $J=g(\bm{x};\dot{W})$. It is a general function of the state vector and the wind shear/gradient, and is subject to the path and boundary constraints. There are many possible objective function candidates for the DS problem as shown in \cite{mir2018review}, however, we only provide here the objective function that is relevant to this work (maximum of energy gain) as follows:
\begin{equation}\label{eqn:objective}
    \begin{split}
        J&=\max(Energy)_{gain},\\
    \end{split}
\end{equation}
The state of the bird at the end of the cycle depends on the mode of DS like basic, loiter and travel modes (see \cite{mir2018review} for more details on modes of DS). So, boundary constraints also vary with different modes of DS. In this work, we use the basic mode and its corresponding boundary constraints:
\begin{equation}\label{eqn:constraints1}
    [z,V,\gamma,\psi]_{t_f}^T=[z+\Delta z, V,\gamma,\psi]_{t_0}^T,
\end{equation}
where $t_f$ and $t_0$ refer to final and initial times respectively and $\Delta z$ is the net change in altitude after the DS cycle. Similarly, path constraints for control and states during the DS cycle are given in \Cref{eqn:constraints2}:
\begin{equation}\label{eqn:constraints2}
    \begin{split}
        V_{\min}<V<V_{\max}, \psi_{\min}<\psi<\psi_{\max},\\
        \gamma_{\min}<\gamma<\gamma_{\max}, x_{\min}<x<x_{\max},\\
        y_{\min}<y<y_{\max},z_{\min}<z<z_{\max} ,\\
        \phi_{\min}<\phi<\phi_{\max}, C_{L_{\min}}<C_L<C_L{_{\max}}.\\
    \end{split}
\end{equation}

\section{Geometric Control Formulation and Analysis of Dynamic Soaring}\label{sec:three_gct}
\subsection{Overview of nonlinear geometric controllability of control-affine systems}\label{sec:gct_ca}

Control-affine systems are defined as those linear or nonlinear systems in which the control inputs appear linearly as in \Cref{eqn:controlAffine}:
\begin{equation}\label{eqn:controlAffine}
    \dot{\bm{x}}=\bm{f}(\bm{x})+\sum_i^m u_i(t) \bm{b}_i(\bm{x}),
\end{equation}
where $\bm{x}\in\mathbb{R}^n$ is the state space vector evolving on an n-dimensional, smooth Manifold $\mathbb{M}^n$, $\bm{f}$ is the uncontrolled vector field of the system (known as the drift vector field), and $\bm{b}_i(\bm{x})$ are the control vector fields associated with $u_i(t)$, the control inputs. Control-affine forms are essential for the application of the GCT \cite{bullo2019geometric,lewis1997configuration,sussmann1987general,lafferriere1991motion} which helps to analyze the controllability, as well as motion planning, of under-actuated nonlinear control systems. Here we focus on controllability, which is one of the most important properties of mathematical and applied control theory. It is usually defined as the ability to steer a dynamical system from a given initial point to a given final point in a finite time, given a number of control inputs. For an under-actuated system (the number of controls is less than the degrees of freedom), this property, practically speaking, determines if a given control system can operate properly based on the available controls. In other words, can we steer a given system with the given controls in any direction we want in its state space? For linear systems of the form,
\begin{equation}\label{eqn:linear}
    \dot{\bm{x}}(t) = \bm{Ax}(t)+\bm{Bu}(t),
\end{equation}
where $\bm{x} \in \mathbb{R}^n $ is the state vector, $\bm{u} \in \mathbb{R}^m$ is the control input vector, $\bm{A} \in \mathbb{R}^{n \times n}$ is the state matrix, and $\bm{B} \in \mathbb{R}^{n \times m} $ is the input matrix. A necessary and sufficient condition for the controllability of the system given by \Cref{eqn:linear} is that the matrix
\begin{equation} \label{eqn:kalmanrc}
    \bm{C} = [\bm{B,AB,A^2B,...,A^{n-1}B}] \in \mathbb{R}^{n \times nm}.
\end{equation}
has full rank, i.e., rank($\bm{C}) = n$ \cite{kalman1963controllability}. This condition, known as Kalman controllability rank condition \cite{kalman1963controllability,nijmeijer1990nonlineardynamical}, is an easy-to-implement condition for linear systems.
%
%
Now, for the controllability of nonlinear systems given by \Cref{eqn:controlAffine}, linearization can be done to have systems in the form of \Cref{eqn:linear} and that can determine local controllability about an equilibrium point by applying Kalman’s criteria in \Cref{eqn:kalmanrc} to the linearized system \cite{nijmeijer1990nonlineardynamical}. However, linearization only offers a sufficient condition for controllability. That is, the actual nonlinear system can be controllable even if the linearized system is not, as shown in the following example.

\textbf{Example 1.} Let us consider a basic ground robot that functions like a baby stroller with $(x_1,x_2) \in \mathbb{R}^2$ as the center of mass and $\theta \in \mathbb{S}^1$ as the orientation angle as shown in \Cref{fig:singleVehicle}. We take an initial position of $\bm{x}_0 = (0,0,0)$. The control inputs for the ground robot are $u_1$ (velocity) and $u_2$ (angular velocity).
%
\begin{wrapfigure}{r}{0.3\textwidth}
  \begin{center}
    \includegraphics[width=0.3\textwidth]{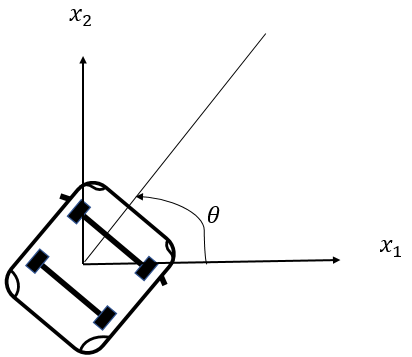}
  \end{center}
  \caption{A basic ground robot.}
  \label{fig:singleVehicle}
\end{wrapfigure}
 This allows for a linear motion $\dot{x}_1 = u_1 cos \theta$ and $\dot{x}_2 = u_1 sin \theta $, and a rotational motion $\dot{\theta} = u_2$. Thus, the system can be formulated as in \Cref{eqn:controlAffine} as
 \begin{equation}\label{eqn:paralleleqn_states}
     \dot{\bm{x}} = \bm{b}_1\bm{u}_1+\bm{b}_2\bm{u}_2, \: \text{with}
 \end{equation}
\begin{align}
    \bm{x}=\begin{bmatrix}
    x_1\\
    x_2\\
    \theta
    \end{bmatrix},
    \bm{b}_1 = \begin{bmatrix}
    cos \theta\\
    sin \theta\\
    0
    \end{bmatrix},
    \bm{b}_2 = \begin{bmatrix}
    0\\
    0\\
    1
    \end{bmatrix}.
\end{align}
This system has 2 control inputs and 3 degrees of freedom, so it is clearly underactuated. Now, the controllability matrix, given by \Cref{eqn:kalmanrc} will have only two vectors resulting from linearization and clearly will not have full rank at $\bm{x}_0$. Thus, it does not satisfy Kalman's rank condition and is linearly uncontrollable. However, from experience, we know the system is controllable. We know also that we cannot ``directly" generate a side motion, which is the cause of linear uncontrollability. However, if one generates certain variations along the available vector fields of the available inputs (linear and angular velocities), we can generate new directions similar to what we experience in parallel parking. This invokes the need for GCT.

The observation provided above in Example 1 is why in literature, nonlinear controllability is mostly associated with GCT.
The literature on nonlinear controllability of nonlinear systems using GCT involves many terminologies and definitions of some weaker and stronger forms of controllability notions, such as accessibility and small-time local controllability (STLC). In addition to these notions, which are typically used to characterize controllability from a fixed point, the concept of local controllability (LC) is used within GCT to characterize controllability along a non-stationary reference trajectory, which will be of particular relevance in this work in section \ref{sec:gct_ds}. For the control-affine system in \Cref{eqn:controlAffine}, we provide formal definitions of these notions \cite[pp. 371]{bullo2004geometric} in Definition 1 that hold throughout this paper. The concepts we provide below in Definition 1 are visually depicted in \Cref{fig:controllability} as a two-dimensional figure.

\textbf{Definition 1 (controllability notions)}: Let the triplet $\Sigma = (\mathbb{M}^n, \mathbb{F},U)$ represent the control affine system in \Cref{eqn:controlAffine}, where $\mathbb{M}^n$ is a smooth manifold, $\mathbb{F} = \{\bm{f},\bm{b_1},...\bm{b_m} \}$ the acting vector fields, and $U$ is the set of admissible controls, which is a function space such that for every $u\in U$, it holds that $u: I \rightarrow U $ for $I \subset \mathbb{R}$ and is locally integrable and bounded. Moreover, the convex hull of $U$ contains the zero vector, $\bm{0_m} \in int(conv(U))$. For a point $\bm{x}_0 \in \mathbb{M}^n$,
\begin{enumerate}[label=D1.\arabic*]
    \item For $T>0, R_{\Sigma} (\bm{x}_0,T)$ denotes the set of points reachable from $\bm{x}_0$ by controlled trajectories of $\Sigma$ in exactly time $T$, i.e., $R_{\Sigma}(\bm{x}_0,T) = \{ \bm{x}(T)\:|\:\bm{x}(t) $ satisfies \Cref{eqn:controlAffine} for all $t \in [0,T]$ with some $\bm{u} \in U$ and $\bm{x}(0)=\bm{x}_0$ \}.
   
    \item  $R_{\Sigma} (\bm{x}_0,\le T)$ denotes the set of points reachable from $\bm{x}_0$ in at most time $T$, i.e.,  $R_{\Sigma} (\bm{x}_0,\le T) = \cup_{t\in[0,T]}  R_{\Sigma} (\bm{x}_0,T)$.
   
    \item The control-affine system $\Sigma$ is said to be accessible from $\bm{x}_0$ if for all $T>0, int(R_{\Sigma}(\bm{x}_0,\le T)) \ne \phi$.
   
    \item The control affine system $\Sigma$ is said to be small-time locally controllable (STLC) from $\bm{x}_0$ if for all $T>0, \bm{x}_0 \in  int(R_{\Sigma}(\bm{x}_0,\le T))$.

    \item The control affine system $\Sigma$ is said to be locally-controllable (LC) about a reference trajectory $\bm{x}^*$ (satisfying \Cref{eqn:controlAffine}) with $\bm{u}=0$ and $\bm{x}^*(0)=\bm{x}_0$) if for some $T>0, \bm{x}^*(t) \in  int(R_{\Sigma}(\bm{x}_0,\le T))$.
\end{enumerate}

\begin{remark} (\cite{bianchini1990graded}, Remark 1.1)
    In the case that the reference trajectory is stationary (equilibrium), then LC about the reference trajectory reduces to STLC.
\end{remark}
\begin{figure}[ht]
    \centering
    \includegraphics[width=0.75\textwidth]{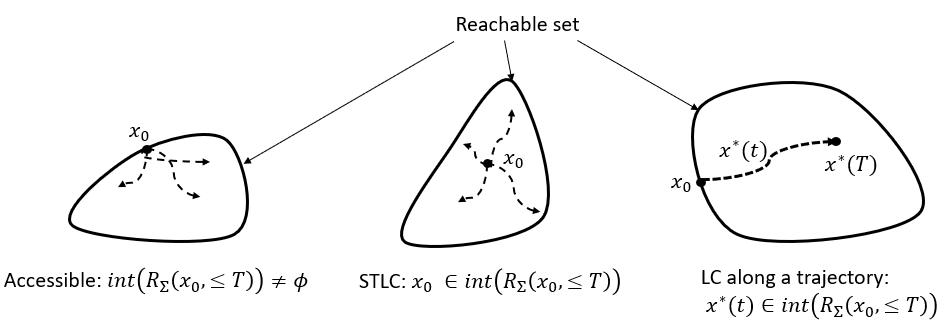}
    \caption{Accessibility, STLC, and LC along a trajectory with relation to the reachable sets.}
    \label{fig:controllability}
\end{figure}

Accessibility is one of the most important and essential properties as it is usually a part of most controllability results or conditions like those in Definition 1. Below is how accessibility is often characterized in GCT \cite{sussmann1972controllability}.
\begin{theorem}\label{thm:larc}
The nonlinear system $\Sigma$, if real analytic, is locally accessible at $\bm{x}_0 \in \mathbb{M}^n$ if and only if the accessibility distribution in \Cref{eqn:larc} has rank $n$.
\begin{equation}\label{eqn:larc}
    \bm{\Delta}_{\bm{x}_0} = span \{ \bm{f}, \bm{b_1},..., \bm{b_m}, [\bm{b}_i,\bm{b}_j],...,ad^k_{\bm{b}_i}\bm{b}_j,...,[\bm{f},\bm{b}_i],...ad^k_{\bm{f}}\bm{b}_j \}_{\bm{x} = \bm{x}_0},
\end{equation}
where the bracket $[.,.]$ is the Lie bracket operation defined below
\begin{equation}
    [\bm{b}_i,\bm{b}_j]:=\frac{\partial \bm{b}_j}{\partial \bm{x}}\bm{b}_i-\frac{\partial \bm{b}_i}{\partial \bm{x}}\bm{b}_j,
\end{equation}
and $ad^k_{\bm{f}}\bm{b}_j = [\bm{f},ad^{k-1}_f \bm{b}_j]$ and $ad^1_{\bm{f}} \bm{b}=[\bm{f},\bm{b}]$. The accessibility condition defined in \Cref{eqn:larc} is called and referred to by the Lie Algebraic Rank Condition (LARC).
\end{theorem}

\textbf{Example 1 (contd.)} Now, we continue Example 1 and analyze the nonlinear controllability using GCT. First, we determine if the accessibility condition in \Cref{eqn:larc} will be satisfied for the system in Example 1, especially since we know from experience that it is controllable and linear controllability has failed due to the inability to steer the system directly in side motion. We calculate the Lie bracket:
\begin{align}\label{eqn:paralleleqn_lie}
[\bm{b}_1,\bm{b}_2]= [\sin \theta, \cos \theta, 0]^T.
\end{align}
Then, $\bm{\Delta_x} = span\{\bm{b_1,b_2},[\bm{b_1,b_2}]\}$ has a full rank of $n=3$ at $\bm{x}_0$, hence the system in  \Cref{eqn:paralleleqn_states} is accessible. Now, for a driftless system, accessibility and controllability are equivalent concepts \cite{bullo2019geometric}. This implies that the system in \Cref{eqn:paralleleqn_states} is controllable. Moreover, the Lie bracket $[\bm{b}_1,\bm{b}_2]$ generates motion perpendicular to the orientation of the robot in which we did not have direct control authority over, and why linear control theory has failed in the first place, see \Cref{fig:multiVehicle}. Simply substituting $\theta=0$ in \Cref{eqn:paralleleqn_lie} shows the side motion direction immediately. In fact, we generate a similar motion (trying to realize the Lie bracket) when we perform parallel parking.
\begin{figure}[ht]
    \centering
    \includegraphics[width=0.8\textwidth]{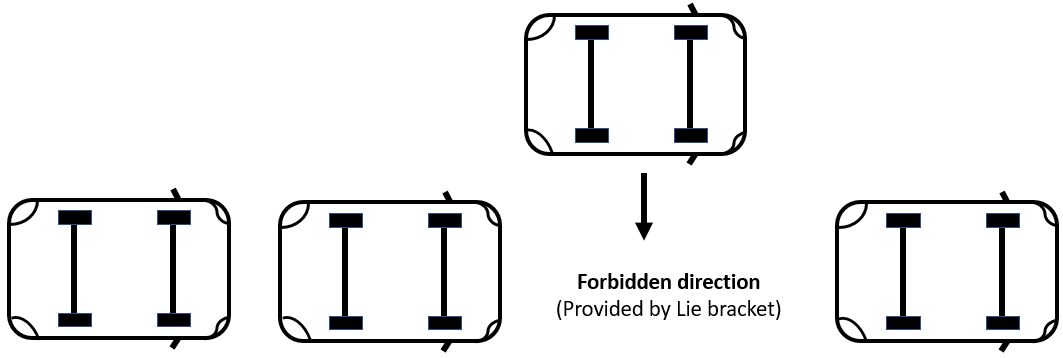}
    \caption{Lie bracket between velocity and angular velocity vector fields generate motion in the forbidden direction.}
    \label{fig:multiVehicle}
\end{figure}
Thus, GCT provides tools to analyze the controllability of systems that are linearly uncontrollable, however, said systems may be nonlinearly controllable.

For systems that have drifts, controllability is not equal to accessibility. That is, even if the system is accessible, it might not be controllable. The main difference between accessibility and controllability is that we may be able to generate motion 
\begin{wrapfigure}{r}{0.4\textwidth}
    \centering
    \includegraphics[width=0.4\textwidth]{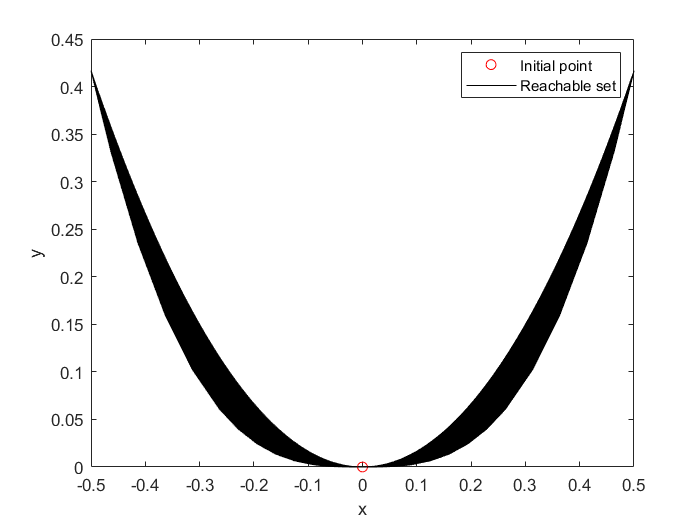}
    \caption{Reachable sets for system in \Cref{eqn:eg2} with $-0.1\le u= constant \le 0.1 $ and $0 \le t \le 5$.}
    \label{fig:reachable_sets_eg2}
\end{wrapfigure}along some direction, but cannot reverse motion along that same direction resulting in the initial point being on the boundary of the reachable set and not in its interior; this is noted from Definition 1 and \Cref{fig:controllability}. This obstruction to local controllability 
is associated with Lie brackets involving an even number of the same control vector $\bm{b}_i$ and an odd number of the drift vector field $\bm{f}$. These Lie brackets are referred to as ``bad brackets" in GCT literature and they are responsible for the said obstruction. 
Because of them, we lack the ability to reverse the flow along the bracket direction by changing the sign of the corresponding control input $u_i$. Given the relevance of the above-mentioned obstruction concept to the results of this paper, we use the following example to explain further the obstruction concept due to ``bad brackets."

\textbf{Example 2: }
In this example, we show that accessibility is not equal to controllability and some Lie brackets called ``bad brackets" obstruct the local controllability. Let us consider a system represented by \Cref{eqn:eg2}:
\begin{align}\label{eqn:eg2}
    \begin{bmatrix}
    \dot{x}\\
    \dot{y}
    \end{bmatrix}=
    \begin{bmatrix}
    u\\
    x^2
    \end{bmatrix}=
    \begin{bmatrix}
    0\\
    x^2
    \end{bmatrix}+
    \begin{bmatrix}
    1\\
    0
    \end{bmatrix} = \bm{f}+\bm{b}u.
\end{align}
Now, to test the accessibility of the system given by \Cref{eqn:eg2} from $\bm{x}_0=(0,0)$, we construct the LARC as in \Cref{eqn:larc},
\begin{align}
    &\bm{b(0)} = \begin{bmatrix}
    0\\
    1
    \end{bmatrix},
    [\bm{f,b}](\bm{0}) = \begin{bmatrix}
    0\\
    -2x
    \end{bmatrix}=\bm{0},
    [[\bm{f,b}],\bm{b}](\bm{0}) = 
    \begin{bmatrix}
    0\\
    2
    \end{bmatrix},\\
    &\bm{\Delta}_{\bm{x}_0} = span \{ \bm{b(0)},[[\bm{f,b}],\bm{b}](\bm{0})  \}.\label{eqn:larc_eg2}
\end{align}

The accessibility distribution in \Cref{eqn:larc_eg2} has a rank of 2, thus, the system is accessible from $\bm{x}_0$. However, the system is not controllable. It can be easily observed that the state $x$ can be directly controlled by the control input $u$; one can reach positive and negative points by changing the control sign. However, the dynamics of the state $y$ is a function of $x^2$ and can only have positive reachable points as demonstrated in \Cref{fig:reachable_sets_eg2} for $-0.1\le u=constant \le 0.1 $ and $0 \le t \le 5$. Thus, the motion in the $y$ direction cannot be reversed and the system is not controllable, but accessible as the origin is on the boundary of the nonempty reachable set. Now, to analyze the concept of obstruction to local controllability within the context of this example, we expand \Cref{eqn:eg2} using the Chen-Fliess series (provided in \Cref{sec:chen}) and analyze it.
Using Chen-Fliess expansion for $y$, we get the exact representation
\begin{equation}\label{eqn:eg2_fliess}
    y(t)= y_0 +x_0^2 t + 2 x_0 \int_0^t \int_0^s u(\sigma) d\sigma ds + 2 \int_0^t \int_0^s u(\sigma) \int_0^{\sigma} u(\tau) d\tau d\sigma ds.
\end{equation}
Given that $\bm{x}_0=(x_0 =0,y_0 = 0)$, $y(t)$ is equal to the last term in \Cref{eqn:eg2_fliess}. No matter what the sign of $u$ is, the state $y$ can only reach points larger than or equal to $y_0=0$. This term remaining in \Cref{eqn:eg2_fliess}  is associated with the Lie bracket $[[\bm{f,b}],\bm{b}]$. Thus, the use of such Lie brackets in LARC causes obstruction to the local controllability, hence the ``bad brackets" terminology.

However, Sussman \cite{sussmann1972controllability} and then Bianchini and Stefani \cite{bianchini1990graded} showed that ``bad brackets" can be, in fact, neutralized in some cases. Their work -- for stationary (fixed point) in the case of Sussman \cite{sussmann1972controllability} and non-stationary (reference trajectory) in the case of Bianchini and Stefani \cite{bianchini1990graded} -- provide a sufficient condition based on a process by which one can determine controllability about a fixed point or reference trajectory, respectively. It comes down to the fact that ``not all bad brackets are the same" and not all of them can preclude controllability. That is, if bad brackets are neutralizable by being linear combinations of some other good brackets, then the system can still be controllable. However, it is not easy to do such a process, and in fact, it has been rarely applied in the literature to real-life applications or major problems, especially \textit{biological systems.} If we recall Example 2 again, we notice that only brackets that include an odd number of $\bm{f}$ and an even number of every $\bm{b}$ are available and necessary for spanning the LARC distribution; this is why the system in Example 2 is only accessible and have to rely on the direction this bad bracket is providing. Nevertheless, this may not be the case in many other systems where bad brackets are not vanishing. For instance, one can find some other good brackets spanning the directions of the identified bad brackets. It turned out, however, that it is not enough to have the bad brackets replaced by a linear combination of good brackets. One also has to show that the alternative good brackets are more dominant. Since bracket length depends on the differentiation order, the more length the bracket has, the less dominant it is. Sussmann’s theorem \cite{sussmann1987general} provides some freedom in defining the length of a bracket and consequently a process for neutralizing bad brackets by other good brackets of smaller length (or weighted length). For the convenience of the reader and the analysis provided in this paper, a simplified (easy to apply) summary of relevant theorems \cite{sussmann1987general,bianchini1990graded} is presented next. However, we provide some basic definitions and notions that allow the application of such theorems. We gathered them in Definition 2 below based on the materials from \cite{sussmann1972controllability,bianchini1990graded,bullo2004geometric}.

\textbf{Definition 2: (notion and process of neutralization of bad brackets)}
\cite[chapter 7]{bullo2004geometric},\cite{sussmann1987general,stefani1985local}
\begin{enumerate}[label=D2.\arabic*]
    \item For a formal bracket $B$, let $|B|_0 $ and $|B|_a$, $a \in \{1,...,m\}$ denote the number of times that vector fields $\bm{f}$ and $\bm{b}_a$ appear in the formal bracket $B$.
   
    \item Let $\beta$ be the set of potential obstructions to STLC, i.e., $\beta$= \{$B: |B|_0$ is odd and $|B|_a $ is even $\forall a \in \{ 1,...,m\}$ \}.
   
    \item Let $\bm{w}=(w_0,w_1,...,w_m)$ be a non-negative weight for each formal bracket $B$ such that $\bm{w}- weight$ of $B$ is defined as $||B||_{\bm{w}} = \Sigma^m_{i=0}w_i|B_i|$.
   
    \item A $\bm{w}-homogeneous$ element $L \in Lie^\infty (\mathbb{F})$ at $\bm{x}_0$ is a linear combination of formal brackets evaluated at $\bm{x}_0$ having the same weight.
   
    \item  A $\bm{w}-homogeneous$ element $L \in Lie^\infty (\mathbb{F})$ is $\bm{w}-neutralized$ at $\bm{x}_0$ if it is a linear combination of formal brackets evaluated at $\bm{x}_0$, with strictly lower $\bm{w}-$ weight.
   
    \item If $S_m$ is the set of permutations on $m$ symbols, $\bm{\sigma} \in S_m$ acts on a weight $\bm{w}$ by acting on the last $m$ components and leaving the first component $w_0$ unchanged, i.e, $\bm{\sigma}(w_0,w_1,...,w_m) = (w_0,w_{\bm{\sigma}(1)}, ...,w_{\bm{\sigma}(m)})$. Then, for a weight $\bm{w}$,$\bm{\sigma} \in S_m^{\bm{w}}$ if $\bm{\sigma}(\bm{w}) = \bm{w}$.
   
    \item For a bracket $B$ and a permutation $\bm{\sigma}, \bm{\sigma} (B)$ is the bracket resulted from applying $\bm{\sigma}$ to each of the input vector fields $\bm{b}_i$'s and leaving the drift vector field unchanged.
   
    \item Define the set $\beta_s^{\bm{w}} = \{ L \in \beta: \bm{\sigma}(L) = L \; \forall \bm{\sigma} \in S_m^{\bm{w}} \}$, i.e., the set of elements that are symmetric with respect to the vector fields having the same weight. Then, the set of $\bm{w}-obstructions$ at $\bm{x}_0$ is defined as $\beta^{\bm{w}} = Lie^{\infty} (\bm{f},\beta_s^{\bm{w}}) \cap S,$ where $S= Lie^{\infty} \{ ad_f^k b_i | i\in \{ 1,...,m\}\}$.
   
    \item If there are no constraints on the size of the controls ($U=\mathbb{R}^m$), then any set of non-negative integers will be admissible as weights. But if $U$ is a hyper-cube, then $w_a \ge w_0 , a \in \{1,..., m \}$ for an admissible weight.

\end{enumerate}

Finally, we provide \Cref{thm:sussman} and \Cref{thm:bianchi} for analyzing the STLC and LC of a control-affine system in the form of \Cref{eqn:controlAffine}.
\begin{theorem}\label{thm:sussman}
(Sussmann's sufficient condition for STLC \cite{sussmann1987general}).\\
Let $\Sigma = (\mathbb{M}^n, \mathbb{F},U)$ be a control-affine system and let $\bm{x}_0 \in \mathbb{M}^n$ be a fixed point. If
\begin{enumerate}
    \item $\Sigma$ satisfies the LARC at $\bm{x}_0$ (\Cref{thm:larc}) and,
    \item there exists an admissible weight $\bm{w}$ such that each element in $\beta^w$ is $\bm{w}-neutralized$ at $\bm{x}_0$,
\end{enumerate}
then $\Sigma$ is STLC at $\bm{x}_0$.
\end{theorem}

\begin{theorem}\label{thm:bianchi}
(Bianchini and Stefani \cite{bianchini1990graded}, Corollary 1.1 ).\\
The system $\Sigma$ is LC along a reference trajectory $\bm{x}^*$ if  \Cref{thm:sussman} holds at $\bm{x}_0 \in \mathbb{M}^n $ with $\bm{x}^*(0) = \bm{x}_0.$
\end{theorem}

\subsection{Novel differential geometric control formulation and controllability analysis of dynamic soaring}\label{sec:gct_ds}
The equations of motion for DS as presented in \Cref{eqn:dynamics} are not in the control-affine form of \Cref{eqn:controlAffine}, hence GCT is not directly applicable to the highly underactuated, nonlinear system of DS. An attempt to introduce a control-affine formulation of the DS problem can be found in the literature by a recent study done by the first author and some other authors \cite{mir2018controllability}. In that work, the body angular velocities $p,q$, and $r$ \cite[chapter 3]{nelson1998flight} were used to transform the DS equations into a control-affine form as follows:
\begin{equation}\label{eqn:DS_CA}
    \dot{\bm{x}} = \bm{f}(\bm{x}) + \bm{b_1}p + \bm{b}_2 q + \bm{b_3}r,
\end{equation}
with $\bm{x} = [x,y,z,V,\gamma,\psi,\phi,\theta]^T \in \mathbb{R}^8$,
using the following transformations:
\begin{align}\label{eqn:transformation}
    p &= \dot{\phi}-\dot{\psi}\sin \theta,\\
     q &= \dot{\theta} \cos \phi + \dot{\psi} \cos \theta \sin \phi, \\
    r &= \dot{\psi} \cos \theta \cos \phi - \dot{\theta} \sin \phi. \label{eqn:constraint_psiphi}
\end{align}

However, a major issue with this formulation is that it is hard to find how $p,q,r$ can be generated as control inputs since they inherently depend on the control inputs in \Cref{eqn:probForm} of the non-control-affine system in \Cref{eqn:dynamics}. More importantly, the dependence between $p,q,r$ can create an issue of well-posedness for the control-affine formulation in \Cref{eqn:DS_CA} as normally we have access to an independent set of admissible control inputs and require $\bm{0}$ to be in the interior of the admissible control set as in Definition 1. For instance, if $r=0$, which was assumed in the above-mentioned work \cite{mir2018controllability}, then \Cref{eqn:constraint_psiphi} would mean that the state dynamics of $\psi$ and $\theta$ are directly related. Since both $\psi$ and $\theta$ are in the state space of the system in \Cref{eqn:DS_CA}, this means that system dynamics is constrained to take place -- regardless of $p$ and $q$ -- in a sub-manifold with dimension $7$ instead of $8$. This invokes the possibility of the Frobenius theorem \cite[chapter 2]{abraham2008frobenius} for the existence of integrable curves which preclude controllability. We demonstrate in the supplementary material through simulation that all of the control inputs $p, q, r$ have to be non-zero in order to obtain equivalent results to the DS system in \Cref{eqn:dynamics} with the control inputs in \Cref{eqn:probForm}. In addition, we could not obtain a similar optimal solution with r = 0.
%
So, in this work, we aim at characterizing the DS problem with a different perspective; instead of using the above-mentioned transformation in \Cref{eqn:DS_CA}-\Cref{eqn:constraint_psiphi}, we propose to use an adaptation/control law for the actual control input(s) that we can generate/apply. In \cite{sachs2005minimum}, it is shown that -- in the DS problem -- one of the control inputs in \Cref{eqn:probForm}, $C_L$, can be taken as a constant while $\phi$ can be the only time-varying control input.  Motivated by the fact that a single roll control input can affect the whole dynamics of DS, we assume the adaptation/control law as $\dot{\phi}=u$. This simple assumption is based on really what soaring birds hypothetically are doing: they change their rolling based on certain assessments and processes, which is the particular control law $u$. This simple assumption leads to a control-affine formulation of the DS problem. However, before moving forward, it is \textit{essential} to ensure that the proposed structure deems the system controllable -- as natural as the bird, especially in the windward climb phase of the DS which is, a critical, yet less studied phase when compared to the gliding (leeward descent) phase \cite{mir2018review}. Furthermore, it is also essential to find that particular control law $u$ that allows the DS conduction and validates our hypothesis of the simplicity of this proposed structure; this will be examined in the next section.

Now we are in a position to utilize GCT to formulate/propose, then analyze and prove the controllability of the DS system. The equations of motion in \Cref{eqn:dynamics} representing the dynamics of birds/mimicking systems performing DS can be written in control-affine form by including $\phi$ in the state of the system itself and having the adaptation law $u$ as the state dynamics. The proposed structure of the control-affine DS system is represented as follows:
\begin{align}\label{eqn:proposedStructure}
    \dot{\bm{x}} = \bm{f}(\bm{x}) + \bm{b}u,
\end{align}
where,
\begin{align}\label{eqn:proposedStructure_states}
    \bm{x} = \begin{bmatrix}
    x\\
    y\\
    z\\
    V\\
    \gamma\\
    \psi\\
    \phi
    \end{bmatrix},
    \bm{f}(\bm{x}) = \begin{bmatrix}
    V \cos\gamma \cos\psi\\
    V \cos\gamma \sin\psi-W\\
    V \sin\gamma\\
   \frac{1}{m}( -D -mg\sin\gamma+m\dot{W} \cos\gamma \sin\psi)\\
    \frac{1}{mV} (L\cos \phi - mg \cos \gamma - m \dot{W} \sin \gamma \sin \psi) \\
    \frac{1}{mVcos \gamma}(L \sin \phi + m \dot{W} \cos \psi) \\
    0
    \end{bmatrix},
    \bm{b} = \begin{bmatrix}
    0\\
    0\\
    0\\
    0\\
    0\\
    0\\
    1
    \end{bmatrix}.
\end{align}

Now, we study the LC of the system along a trajectory $x^*(t)$. This reference trajectory is defined by setting $u=0$ \cite{bianchini1990graded}. Remarkably, this is exactly the DS system and equations of motion in \Cref{eqn:dynamics} with a constant rolling; the bird then updates the actuation based on the control law (its assessment and sensing of its state) and is expected to be controllable as it is in nature. According to \Cref{thm:bianchi}, LARC  needs to be computed only at the initial point of the trajectory. We choose the initial point to be in the climb phase responsible for acquiring lift (energy from the wind) and height (have the orientation of the body in headwind positioning). The initial point is taken as
\begin{equation}\label{eqn:initial}
    \bm{x}_0 = [0,0,z_0,V_0,\gamma_0,\psi_0,\phi_0]^T
\end{equation}
with $ z_0 = 10,V_0=14,\gamma_0=0.6,\psi_0=1.4,\phi_0=-0.1 $ with units in SI system. Now, we present \Cref{thm:controllability} to provide LC along the reference trajectory $x^*(t)$.
\begin{theorem}\label{thm:controllability}
The control-affine system in \Cref{eqn:proposedStructure}-\Cref{eqn:proposedStructure_states} characterized by the triplet $\Sigma = (\mathbb{M}^7,\mathbb{F}= \{\bm{f,b} \},U \text{as in Definition 1})$ is locally controllable along the reference trajectory $\bm{x}^*(t)$ with $\bm{x}^*(0)=\bm{x}_0$ in \Cref{eqn:initial}. 
\end{theorem}

\begin{proof}
First, we verify the LARC distribution in \Cref{eqn:larc} as follows:
\begin{align}\label{eqn:span_in_proof}
    \Delta _{\bm{x}_0}= span\{ \bm{f},  \bm{b}, [\bm{f},[\bm{f},\bm{b}]], [\bm{f},[\bm{f},[\bm{f},\bm{b}]]], [\bm{f},[\bm{f},[\bm{f},[\bm{f},\bm{b}]]]], [\bm{f},[\bm{f},[\bm{f},[\bm{f},[\bm{f},\bm{\bm{b}}]]]]]  \}_{\bm{x}_0}.
\end{align}
The rank of ($\Delta _{x_0})=7$ which is full. Thus, the system given by \Cref{eqn:proposedStructure} satisfies the first requirement of \Cref{thm:sussman} and \Cref{thm:bianchi}. No bad brackets -- element of $\beta$ -- are utilized in ensuring the LARC. However, because the bad bracket $B_b = [\bm{b}, [ \bm{f},\bm{b}]]$ does not vanish at $\bm{x}_0$, further analysis (Definition 2) is required to ensure that the bad bracket $B_b$ is neutralizable and does not preclude controllability.
Since $|B_{b}|_0 = 1$ and $|B_b|_1 =2$, then $B_b \in \beta$.  Now, by choosing $\bm{w} = (w_0 ,w)$ for some $w_0,w \in \mathbb{Z}^+$, it is clear that $||B_b||_{\bm{w}}= w_0+2w$. Furthermore, we have $L = B_b$ with $\bm{\sigma}(L) = L$. This implies that $\beta_s^w = \{ L \} \subseteq \beta^w$. Clearly $L \in span \{ \Delta_{\bm{x}_0} \}$ in \Cref{eqn:span_in_proof}. $L$ is neutralizable if it has $\bm{w}-$weight strictly larger that $\bm{w}-$weight of the good brackets used in  $\Delta_{\bm{x}_0} $ in  \Cref{eqn:span_in_proof}. It follows immediately that the quantity $w_0 +2w$ is strictly larger than the weights $w_0, w, w+w_0$ of $\bm{f},\bm{b},[\bm{f},\bm{b}]$. However, in the cases of $[\bm{f},[\bm{f},\bm{b}]], [\bm{f},[\bm{f},[\bm{f},\bm{b}]]], [\bm{f},[\bm{f},[\bm{f},[\bm{f},\bm{b}]]]]$, and $[\bm{f},[\bm{f},[\bm{f},[\bm{f},[\bm{f},\bm{b}]]]]]$, following inequalities need to be satisfied.
\begin{align}
    w_0+2w &> 2w_0+w = ||[\bm{f},[\bm{f},\bm{b}]] ||_{\bm{w}},\\
    w_0+2w &> 3w_0+w =|| [\bm{f},[\bm{f},[\bm{f},\bm{b}]]]||_{\bm{w}},\\
    w_0+2w &> 4w_0+w =||  [\bm{f},[\bm{f},[\bm{f},[\bm{f},\bm{b}]]]]||_{\bm{w}},\\
    w_0+2w &> 5w_0+w =||  [\bm{f},[\bm{f},[\bm{f},[\bm{f},[\bm{f},\bm{b}]]]]]||_{\bm{w}}.
\end{align}
All of the above inequalities are simultaneously satisfied if $w>4w_0$ (for example, take $\bm{w} = (1,5)$). We need, however, to ensure the neutralization of the other elements in $\beta^w = Lie^\infty (\bm{f},L) \cap S$. For an arbitrary element $B_{\beta} \in \beta^w$,  we have $B_{\beta} \in Lie^\infty (\bm{f},L) \backslash \{\bm{f} \}$. It follows immediately that $|B_{\beta}|_{\bm{w}} \ge |{L=B_b}|_{\bm{w}} = w_0 +2w$. This implies that any $B_{\beta} \in \beta^w$ is neutralizable by the brackets in $\Delta_{\bm{x}_0}$ if $w>4w_0$. This, in addition to satisfying the LARC, proves LC along the trajectory $x^*(t)$.
\end{proof}

\section{Dynamic soaring as a control-affine extremum seeking system}\label{sec:four_esc}
\subsection{Results and simulations}\label{sec:simulations}

Here, we propose the expression of the particular control law $\dot{\phi} = u$, which effectively turns the control-affine DS system in \Cref{eqn:proposedStructure_states} into a control-affine ESC system similar to those provided in \cite{ESCTracking,DURR2013}; this DS ESC system, like any ESC system, can be model-free and is implementable \textit{in real time}.
A common structure for a control-affine ESC system is \cite{ESCTracking,DURR2013,VectorFieldGRUSHKOVSKAYA2018}:
\begin{equation}\label{eqn:ESC2018}
    \dot{x}=u,
\end{equation}
where
$u=b_1(J(x))u_1(t)+ b_2(J(x)) u_2 (t);$ $b_1,b_2$ are functions of $J(\bm{x})$ itself (the objective function).
It is important to emphasize that $J(x)$ can be accessible through measurements and not necessarily by a mathematical expression as it is known in ESC literature in general \cite{ariyur2003real,DURR2013}. For $b_1,b_2,u_1$ and $u_2$, the following assumptions need to be satisfied.
\begin{enumerate}[label=A\arabic*.]
    \item 
    $b_i\in C^2: \mathbb{R} \to \mathbb{R}, i={1,2}$.
    \item 
    For every compact set $\mathscr{C} \subseteq \mathbb{R}$, there exist $A_1, ..., A_3 \in [0,\infty)$ such that $|b_i(x)|\leq A_1,
    | \frac{\partial b_i(x)}{\partial x}|\leq A_2,
    |\frac{\partial [{b_j},{b_k}](x)}{\partial x}| \leq A_3$ for all $x\in \mathscr{C}, i={1,2};\; j={1,2};\; k={1,2}.$
    \item 
    ${u}_i: \mathbb{R} \times \mathbb{R} \to \mathbb{R} , i=1,2$, are measurable and bounded functions. Moreover,
    ${u}_i(\cdot)$ is T-periodic, i.e. ${u}_i(\omega t + T)={u}_i(\omega t),$ and has zero average, i.e. $\int_0^T {u}_i(\tau) d\tau = 0,$ with $T \in (0,\infty)$ for all $\omega t \in \mathbb{R}$.
\end{enumerate}
A particular case of control-affine ESC systems, typically used in many ESC problems \cite{DURR2013,VectorFieldGRUSHKOVSKAYA2018} is when we choose $b_1 = J(x)$, $b_2 = 1$ as demonstrated in \Cref{fig:control_affine}.  This ESC structure, via the variations of $u_1$ and $u_2$, realizes the gradient based on the objective function measurements and converges to a neighborhood around the extremum point of the objective function. In fact, $u_1$ and $u_2$ are better to be orthogonal to each other. The signals $u_1$ and $u_2$ are typically taken as sinusoidal functions. In fact, in GCT language, setting $u_1$ and $u_2$ as sinusoidal functions provide a distribution for the LARC (see \Cref{thm:larc}) that is $\Delta$=span\{$[J(x),1]=-\nabla J(x)$, i.e., the tangent space of this simple ESC system in \Cref{eqn:ESC2018} can be spanned effectively by the gradient depending on the frequency as known in control-affine systems. For more 
details on this, we provide a demonstrating example and simulation in the supplementary material. Our proposed ESC system for DS will also have a similar structure to \Cref{fig:control_affine} but the cost function will depend on the states of the DS system as shown in \Cref{fig:framework}.

As it is seen in nature, birds like albatrosses have the energy for flying hundreds of miles by performing DS maneuvers autonomously and in real-time. So, there should be a real-time autonomous control mechanism that enables the birds to find the path allowing them to extract maximum energy from wind; hence, reducing the flying effort. One way to do that is to perturb about its state by a control action and assess -- through sensing -- the corresponding change in wind, speed, and altitude. Based on the energy state, the actuation decision will be taken. Now, if assessing the energy state is a 
periodic/periodic-like process, then the actuation becomes much easier to conduct and learn (memorize). Our proposed ESC structure does the same process in real-time. An objective function $J(\bm{x},\dot{W})$ is designed
\begin{wrapfigure}{r}{0.4\textwidth}
    \centering
    \includegraphics[width=0.4\textwidth]{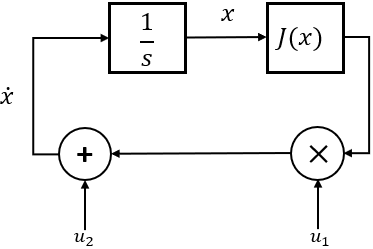}
    \caption{Structure of a particular case of control-affine ESC systems.}
    \label{fig:control_affine}
\end{wrapfigure}
  such that the changes are measured (sensed) based on wind, velocity, and height. The use of control signals $u_1$ and $u_2$ provides a periodic repetitive actuation that incorporates the objective function measurements. Assumptions A1-A3 have to be ensured so that the control input $u$ will cause autonomous seeking of maximization of the objective function. Now, for the objective functions in \Cref{eqn:objective}, we provide the corresponding expression to specific energy gain, $J=-V\dot{W} \cos \gamma \sin \psi /g $ aiming at maximizing it. This expression is derived here. We define specific energy as total energy per weight, such that:
\begin{equation}
    e = E/mg =z+\frac{V^2}{2g}.
\end{equation}
Now, taking derivative with respect to time, we have:
\begin{equation}
\begin{split}
    \dot{e}&=\dot{z}+\frac{V\dot{V}}{g}= V \sin\gamma + \frac{V}{g}[-D-g\sin\gamma -\dot{W}\cos \gamma \sin \psi] \\
    &=-\frac{DV}{g}-\frac{V\dot{W}\cos \gamma \sin \psi}{g}.\label{eqn:energyGain}
\end{split}
\end{equation}
Observe that from \Cref{eqn:energyGain}, the term $-V\dot{W} \cos \gamma \sin \psi /g $ is solely responsible for energy gain from the wind. Hence, we use said term as the objective function.
\begin{figure}[ht]
    \centering
    \includegraphics[width=0.7\textwidth]{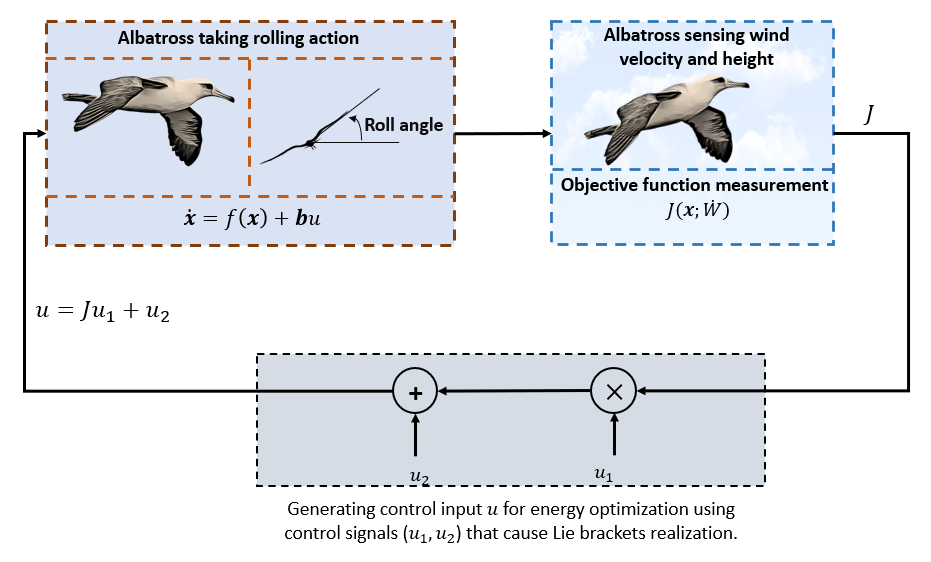}
    \caption{Framework for ESC control structure for the DS problem.}
    \label{fig:framework}
\end{figure}
\begin{figure}[ht]
\centering
\includegraphics[width=0.7\textwidth]{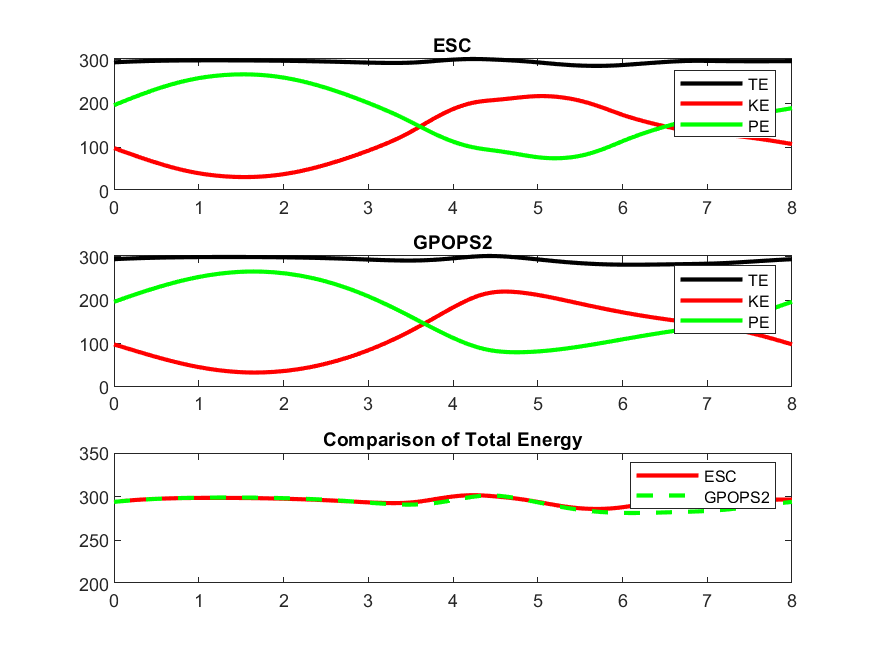}
\caption{Comparison of the total energy of albatross obtained using ESC controller and GPOPS2.}
\label{fig:TotalEnergy}
\end{figure}
Now, we present the simulation results and compare them with the optimal control software GPOPS2 \cite{gpops2} which uses hp-adaptive Gaussian quadrature collocation methods and sparse nonlinear programming. The Simulink\textsuperscript{\textregistered} model used for our control structure is provided in the supplementary material.  The characteristics of the bird and wind in the simulation are taken as in \Cref{tab:params}. The control signal $u_1$ is taken as $a \cos (\omega t + \mu)$ and $u_2$ is taken as $b \sin (\omega t)$ with $\omega = 5.8, \mu = 0.55, a = 2.1, b = 0.8 $. We have carefully tuned the parameters above for optimal results. The initial states taken for simulation is $ [0, 0, 10, 14, 0.6 ,1.4, -0.1]$ and is the same as the initial point used in \Cref{thm:controllability}. \Cref{fig:TotalEnergy} shows the total energy level during the DS maneuver obtained using our control-affine ESC implementation and GPOPS2 non-real-time implementation, and a comparison between the two methods. It can be seen in \Cref{fig:TotalEnergy} that the total energy (TE) throughout the DS cycle is almost constant. Moreover, the total energy is almost the same at the beginning and the end of the cycle, i.e., the kinetic energy (KE) and the potential energy (PE) are traded off during the maneuver without any loss to drag. In other words, the system is a conservative-like system, contrasting with typical flight dynamic systems which are dissipative due to the non-conservative drag force. This highlights that our implementation is successful as we have achieved the most significant characteristic of the DS phenomenon: energy neutrality (or near neutrality). \Cref{fig:3D} compares the trajectory of the bird obtained using our proposed control-affine ESC in real-time with the optimal trajectory obtained using GPOPS2 in non-real-time starting at the same initial point.
The trajectory obtained using the DS control-affine ESC consists of all of the characteristic phases of DS, i.e., wind-ward climb, high altitude turn, leeward descent, and low altitude turn, and the trajectory is also very close to the optimal trajectory provided by GPOPS2. \Cref{fig:allStates} provides the state variables and control input vs. time using the proposed control-affine ESC in real-time vs.GPOPS2 results computed in non-real-time.
\begin{figure}[ht]
\centering
\includegraphics[width=0.8\textwidth]{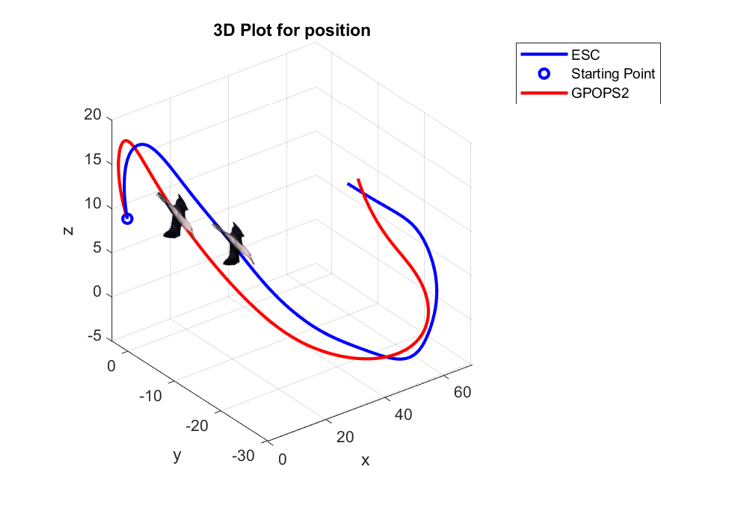}
\caption{Comparison of trajectories of albatross obtained using ESC controller and GPOPS2.}
\label{fig:3D}
\end{figure}
The trajectories of the states obtained using both control-affine ESC and GPOPS2 are close. Here we summarize the novel advantages of using ESC vs. optimal optimizers like those in DS decades-long literature. Our proposed DS, ESC system is autonomous and operates in real-time, without the need for a model or expression for the objective function. Moreover, the ESC implementation does not need the constraints and bounds in \Cref{eqn:constraints1} and \Cref{eqn:constraints2}. On the contrary, the optimal control solver GPOPS2 representing the decades-long optimal control configuration of DS, requires all provided models, bounds/constraints, specific wind profile models in section \ref{sec:two_prob_form}, and still operates in non-real-time; at times, optimal control solvers can take up to 1000 seconds to perform DS \cite{mir2018optimal}. Additionally, the proposed control-affine ESC system is much simpler and appears to be very natural to this nature-inspired phenomenon suggesting that the provided work in this paper captures important aspects of the biological phenomenon under study; it also takes us steps closer to real and industrial mimicking of soaring birds.


\subsection{Extremum seeking systems manifestation in nature: our approach and other's}\label{sec:discussion}
We dedicate this subsection to discussing the significance of our results and others who proposed similar ideas and contributions. In other words, we here like to put this paper in the context of an effort that has been taking place for more than a decade in decoding biological behavior using simpler mathematical methods, particularly ESC systems. Interestingly, proposing, decoding, or analyzing biological behaviors via easy and simple methods such as ESC, is not free from advanced mathematics and analysis. That is, simple methods that achieve strong conclusions/characterizations in biology seem to be analyzed/proved via -- while beautiful -- advanced and complex mathematical tools. The advantage here though is that the beautiful and advanced mathematics in simple methods is used to prove and analyze the dynamical behavior and its potential, but not for the implementation itself as this often contradicts the biological organism's computation/comprehension abilities. Starting with the paper at hand, one can see that it takes decades of advanced differential geometric control theory (recall section \ref{sec:three_gct}) to prove the controllability of the simple proposed control-affine DS system in \Cref{eqn:proposedStructure_states}. Hence, one can be assured that the very effective, but very simple, control-affine ESC system representing DS is able to execute the maneuver. We here recall the inspirational work in \cite{fish1}, which presents a biologically plausible interpretation of fish seeking prey without the knowledge of its position as well as the position of the source, in a model-free fashion, using also control-affine ESC. Even though a complicated fluid-structure interaction is involved in fish swimming, and there is  the presence of dynamics governed by the Navier-Stokes partial differential equations, the presented methods in that work are simple enough for a fish-brain implementation. This successful implementation of ESC methods for replicating fish swimming in real-time without involving the complexity of the fluid-structure, is almost conceptually paralleled in the problem under study in this paper: replace the fish with a bird, and the complex fluid-structure with the complex aerodynamics. Similar parallelism can be found in other works in literature, for example, Figure 2 in \cite{Cochran2009autonomous} and Figure 3 in \cite{cochran2009_3dsource_seeking} can be compared to our structure in \Cref{fig:framework}. So, the ESC implementation of DS implies that without comprehending the wind models and relevant aerodynamic complexities, the ESC process seems simple enough for a bird-brain to work, which then strengthens the plausibility of DS being an ESC manifestation in nature. Similarly, \cite{liu2010stochastic} provides an interpretation of chemotaxis-like motion observed in the bacterium Escherichia coli (\textit{E. Coli}) using stochastic source seeking algorithm in a control-affine ESC system framework. In that work, the authors consider a nonholonomic unicycle model that has no knowledge of its position, nor of the distribution of the signal field. The system reaches the optimal position by only sensing the signal locally; note here the conceptual parallelism between sensing said signals and albatross sensing wind speed via their nostrils as discussed in sections \ref{sec:motivation} and \ref{sec:contribution}, and provided in \Cref{fig:framework}. Even though advanced mathematical tools such as stochastic averaging are needed for the analysis and mathematical proof in said work, the ESC method itself -- similar to our DS control-affine ESC method -- is very simple and provides a plausible explanation for the motion of a simple organism like bacteria.

\section{Conclusions}
\label{sec:conclusions}
An effort towards building a bridge between the physics and biology of soaring birds is provided in this paper utilizing mathematical control theory and tools from geometric control theory and extremum seeking methods. The provided effort leads to a proposed system and control structure that replicates and implements in \textit{real-time} the DS phenomenon. This work is categorically different from the decades-long literature on DS. Hence, we hypothesize that is plausible that the proposed work is similar/close to the biological behavior observed in soaring birds. Expanding more on this, the paper provides a detailed mathematical analysis of a simple DS control-affine system using decades of mathematical advancements in GCT; the proposed control-affine DS system is proven to be controllable based on rolling action alone. Moreover, we provided a particular adaptation/control law for the rolling action inspired by ESC systems which effectively turned the control-affine DS system into a DS, control-affine ESC system. ESC philosophy, simplicity, and applicability in real-time in a model-free fashion, has significant similarity in their process/method to what is believed to be done by soaring birds: a process of dynamic optimization in real-time based on sensing (measurement) of their states to achieve a desired optimal state (maximizing energy gain from wind). The implemented framework in this paper, not only narrows the gap between the physics and biology of soaring birds using beautiful and advanced mathematical tools, but it also provides the control engineering literature with a more practical step towards industrial mimicking of soaring birds by unmanned systems and robotics. As discussed in the paper, this work continues an effort that has been taking place for over a decade to propose simple, but plausible, mathematical methods for decoding/mimicking biological behaviors that are often done by simple organisms that are unable to conduct complex computations, and certainly, cannot be conducting non-real-time motions.
\begin{figure}[ht]
\centering
\includegraphics[width=\textwidth]{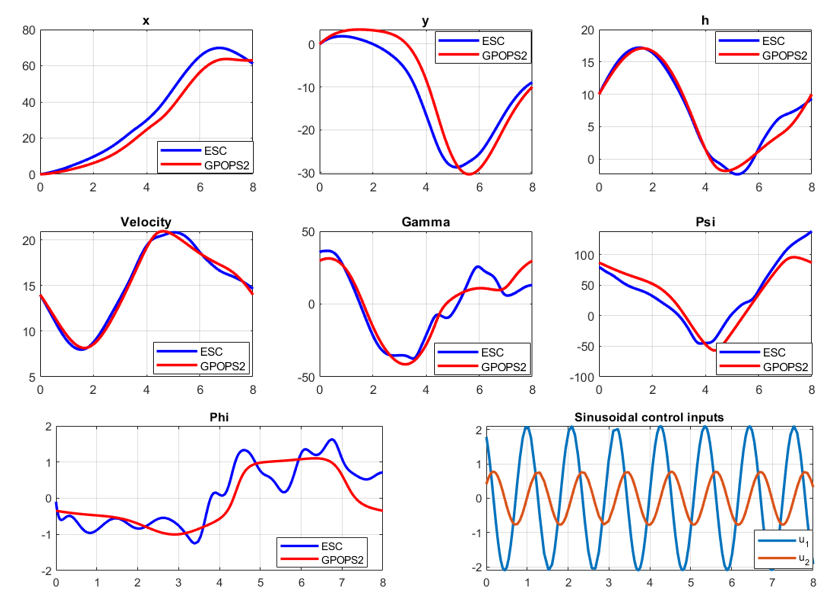}
\caption{Comparison of states and control inputs for albatross obtained using ESC controller and GPOPS2.}
\label{fig:allStates}
\end{figure}

\section*{Acknowledgement}
We would like to thank the editor, Prof. Suncica Canic, for the time and effort spent in conducting the review of this paper. In particular, the authors are grateful for the very helpful comments provided by the very experienced reviewer. This has helped to strengthen the case for the paper and put it in context with the literature. 

\appendix
\section{Chen-Fliess series}\label{sec:chen}
Chen-Fliess functional expansion \cite[chapter 3]{isidori1985nonlinear} is used to describe representations of input-output behavior of a nonlinear system described by differential equations of the control-affine form in \Cref{eqn:controlAffine} and associated output function $y=h(\bm{x})$, $y\in \mathbb{R}$. Let $T$ be a fixed value of the time and $u_1,...,u_m$ are real-valued piecewise continuous functions defined on [0,T]. Now, iterated integral for each multi-index ($i_k,...,i_0$) is defined as
\begin{equation}
    \int_0^t d\xi_{i_k}...d\xi_{i_0}=\int_0^t d\xi_{i_k}(\tau) \int_0^t d\xi_{i_{k-1}}... \int_0^t d\xi_{i_0},
\end{equation}
where $0\le t \le T$; $\xi_0(t)=t$; $\xi_i(t)=\int_0^t u_i(\tau)d\tau$ for $1 \le i \le m$.
%
Now, the Chen-Fliess expansion provides the evolution of the output $y(t)$ for a short time $t\in[0,T]$ as

\begin{equation}\label{eqn:chenFliess}
    y(t)=h(\bm{x}_0)+\sum \limits_{k=0}^\infty \sum \limits_{i_0 , ..., i_k = 0}^m  L_{\bm{b}_{i_0}}... L_{\bm{b}_{i_k}} h(\bm{x}_0) \int_0^t d\xi_{i_k}...d\xi_{i_0},
\end{equation}
where $L_{\bm{b}}h$ is the Lie derivative of $h$ along $\bm{b}$, i.e. $L_{\bm{b}}h=\nabla h\cdot \bm{b}$.

\section{Lie brackets}
We provide here the first two Lie brackets used in LARC distribution in \Cref{eqn:span_in_proof} between vector fields $\bm{f}$ and $\bm{b}$. The reader can find the generating code for all Lie brackets used in our results in the supplementary material.
\begin{equation*}
    [\bm{f},\bm{b}] = \left[ 0, 0, 0,0, \frac{L S_\phi }{mV}, \frac{-L C_\phi \sec (\gamma) }{ mV}, 0 \right]^T,
\end{equation*}

\begin{equation*}
    [\bm{f},[\bm{f},\bm{b}]] = \begin{bmatrix}
     \frac{L}{m}(C_\psi S_\gamma S_\phi - C_\phi S_\psi )\\
     \frac{L}{m} (C_\phi C_\psi + S_\gamma S_\phi S_\psi)\\
     -\frac{L}{m}C_\gamma S_\phi\\
     \frac{\eta L W_0 C_\phi C_\psi S_\gamma }{\delta \xi^2 m }- \frac{L S_\phi }{m^2 V} \left (-mg C_\gamma + \frac{\eta m V W_0 S_\phi }{\delta \xi^2}(C_\gamma^2-S_\gamma^2)\right )\\
     -\frac{L S_\phi}{m^2V^2} \left(  - \frac{ \eta m V W_0 C_\gamma S_\gamma S_\psi}{\delta \xi^2} -D   \right )- \frac{\eta L W_0 c_\phi C_\psi S_\gamma T_\gamma }{mV\delta \xi^2}   \\
     \frac{LC_\phi sec(\gamma) }{m^2V^2} A -\frac{\eta LW_0 C_\phi sec(\gamma) S_\phi T_\gamma}{mV \xi^2 \delta} -\frac{L S_\phi }{mV}B\\
     0
    \end{bmatrix},
\end{equation*}
with 
$C_{(\cdot)} = cos(\cdot)$, $S_{(\cdot)} = sin(\cdot)$, $T_{(\cdot)} = tan(\cdot)$, $\eta = e^{-h/\delta}$, $\xi = (1+e^{-h/\delta})$,
\begin{align*}
    A &= -D -mgS_\gamma + \frac{\eta m V W_0 C_\gamma S_\gamma  S_\psi}{\xi^2 \delta} + T_\gamma \left (mgC_\gamma -L C_\phi +\frac{\eta m V W_0 S_\gamma^2 S_\psi}{\xi^2 \delta} \right ),\\
    B &= \frac{\eta W_0 C_\psi }{ \xi^2 \delta}+ \frac{sec(\gamma)T_\gamma}{mV} \left (\frac{\eta m V W_0 C_\psi S_\gamma }{\xi^2 \delta }+ L S_\phi \right).
\end{align*}

\bibliographystyle{siamplain}
\bibliography{references}

\pagebreak
\section{Supplementary materials}
\subsection{Comparison of two Dynamic Soaring Structures}
In this section, we present a comparative result between two dynamical models for dynamic soaring. The first model, referred as ``DS", has the following state vector $\bm{x}(t)$ and the control vector $\bm{u}(t)$ as
\begin{equation}\label{eqn:probForm_app}
\begin{split}
    \bm{x}(t)&=[x,y,z,V,\gamma,\psi],\\
    \bm{u}(t)&=[C_L,\phi],
\end{split}
\end{equation}
and the nonlinear system is in the following form
\begin{equation}
    \bm{\dot{x}}(t)=\bm{f}(\bm{x}(t),\bm{u}(t)).
\end{equation}
Now, the second model, referred to as "control-affine DS" is a control-affine dynamic soaring model similar to \cite[section 5]{mir2018controllability} in the following form
\begin{equation}\label{eqn:DS_CA_app}
    \dot{\bm{x}} = \bm{f}(\bm{x}) + \bm{b_1}p + \bm{b_2}q + \bm{b_3}r,
\end{equation}
with
\begin{align}
    \bm{x}(t) &= [x,y,z,V,\gamma,\psi,\phi,\theta]^T,\\
    \bm{u}(t) &= [p,q,r].
\end{align}

We have simulated both models to obtain the optimal trajectory as shown in \Cref{fig:3D_plot_comparison}.  We can see in \Cref{fig:pqr} that all of the control inputs $p,q,r$ are non-zeros, i.e., all of the control inputs are essential to obtain the optimal trajectory (as shown in \Cref{fig:3D_plot_comparison}). In addition, we could not obtain a similar optimal solution with $r=0$. This verifies that $r=0$, as assumed  in \cite{mir2018controllability}, constraints the system dynamics  in a sub-manifold with dimension 7 instead of 8, and is not recommended in controllability analysis for the system in the form \Cref{eqn:DS_CA_app}.
\begin{figure}[ht]
    \centering
    \includegraphics[width=0.9\textwidth]{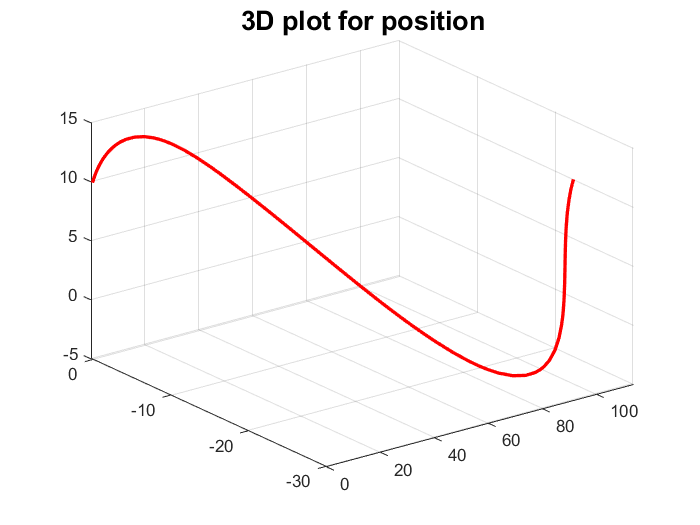}
    \caption{3D trajectories for the dynamic soaring problem generated using $\bm{u}(t)$ in \Cref{eqn:probForm_app} or equivalently $\bm{u}(t)$ in \Cref{eqn:DS_CA}.}
    \label{fig:3D_plot_comparison}
\end{figure}
\begin{figure}[h]
    \centering
    \includegraphics[width=0.8\textwidth]{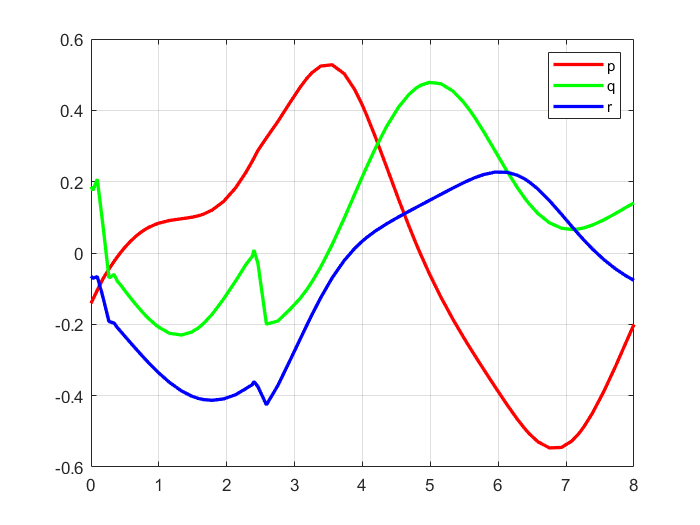}
    \caption{Body angular velocities.}
    \label{fig:pqr}
\end{figure}

\clearpage

\subsection{Control-affine Extremum Seeking Control}
In this section, we use a common structure for a control-affine Extremum Seeking Control (ESC) system in the form
\begin{equation}\label{eqn:ESC2018_app}
    \dot{x}=u,
\end{equation}
where $u=b_1(J(x))u_1(t)+ b_2(J(x)) u_2 (t);$ $b_1,b_2$ are functions of the objective function $J(\bm{x})$, to show the working of the control affine  ESC. As it is typically used in many applications like in \cite{DURR2013,VectorFieldGRUSHKOVSKAYA2018}, we set $b_1 = J(x)$, $b_2 = 1$. We take $u_1= \sqrt{\omega} \cos(\omega t)$, $u_2=\sqrt{\omega} \sin(\omega t), \omega=50$,  with $J(x)=2(x-x^*)^2, x\in \mathbb{R}; x^*=1, a=1$. The plot in  \Cref{fig:ESC_example} shows the evolution of the system \Cref{eqn:ESC2018_app} with time. The system is able to find an optimal point it is orbiting around it.
\begin{figure}[ht]
    \centering
    \includegraphics[width=0.5\textwidth]{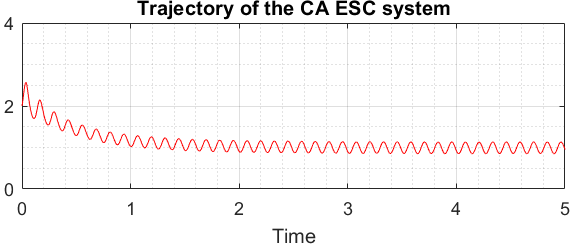}
    \caption{Trajectory of the CA ESC system represented by \Cref{eqn:ESC2018_app}.}
    \label{fig:ESC_example}
\end{figure}

\subsection{Simulink Diagram for Control-affine ESC controller}
In this section, we present the structure of the ESC system used as a real-time autonomous controller for dynamic soaring. The model of dynamic soaring takes two control inputs $C_L$ and $\phi$. Motivated by the work \cite{sachs2005minimum}, where it is shown that in the DS problem, one of the control inputs, $C_L$, can be taken as a constant while $\phi$ can be the only time-varying control input, we have also fixed $C_L$ and have an adaptation/control law for the other control input $\phi$. That adaptation law is provided by control-affine Extremum Seeking Control (ESC) as shown in \Cref{fig:simulinkModel}.
\begin{figure}[ht]
    \centering
    \includegraphics[width=0.9\textwidth]{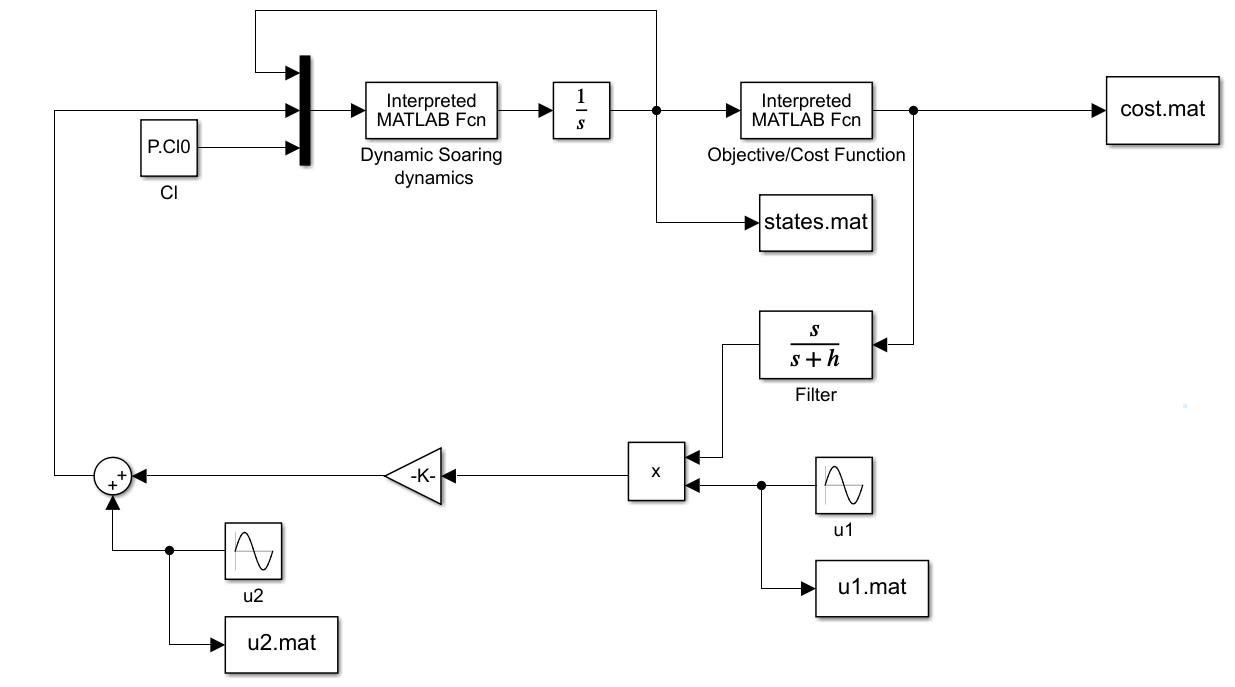}
    \caption{Simulink model for control-affine ESC controller.}
    \label{fig:simulinkModel}
\end{figure}

\clearpage

\subsection{Mathematica Codes for Lie bracket computation}
In this section, We have included the Mathematica codes to obtain the Lie brackets used in the following distribution  to verify LARC.

\begin{align*}\label{eqn:span_in_proof_app}
    \Delta _{\bm{x}_0}= span\{ \bm{f},  \bm{b}, [\bm{f},[\bm{f},\bm{b}]], [\bm{f},[\bm{f},[\bm{f},\bm{b}]]], [\bm{f},[\bm{f},[\bm{f},[\bm{f},\bm{b}]]]], [\bm{f},[\bm{f},[\bm{f},[\bm{f},[\bm{f},\bm{\bm{b}}]]]]]  \}_{\bm{x}_0}.
\end{align*}

\textbf{Codes for Mathematica}

\text{Clear}[\text{Global$\grave{ }$*}]

\begin{doublespace}
\noindent\(\pmb{\text{}}\)
\end{doublespace}

\text{(*Logistic Wind Model*)}\\
\text{vw} = \text{W0}/(1 + \text{Exp}[-h/\text{$\delta$}]); \text{(* Wind Velocity*)}\\
\text{vwDot} = \text{W0} \text{Exp}[-h/\text{$\delta$}] V \text{Sin}[$\gamma$ ]/ (\text{$\delta$} (1 + \text{Exp}[-h/$\delta])^{\wedge}2)$;

\begin{doublespace}
\noindent\(\pmb{\text{}}\)
\end{doublespace}

\text{(*Provide energy gain cost function*)}\\
J= -V \text{vwDot} \text{Cos}[$\gamma$ ] \text{Sin}[$\psi$ ]/g;

\begin{doublespace}
\noindent\(\pmb{\text{}}\)
\end{doublespace}

\text{(*}\text{Write} \text{the} \text{states}, \text{vector} \text{fields} \text{for} \text{the} \text{system}\text{*)}\\
\text{stateX}=\{\text{xpos},\text{ypos},h,V, $\gamma$ , $\psi$ ,$\phi$ \};\\
\text{f1}= V \text{Cos}[$\gamma$ ] \text{Cos}[$\psi$ ] ;\\
\text{f2}=V \text{Cos}[$\gamma$ ] \text{Sin}[$\psi$ ]-\text{vw};\\
\text{f3}=V \text{Sin}[$\gamma$ ];\\
\text{f4}=1/m (-\text{drag}-m g \text{Sin}[$\gamma$ ] + m \text{vwDot} \text{Cos}[$\gamma$ ] \text{Sin}[$\psi$ ]);\\
\text{f5}=(1/(m V ))(\text{lift} \text{Cos}[$\phi$ ]-m\text{  }g \text{Cos}[$\gamma$ ] - m \text{vwDot} \text{Sin}[$\psi$ ] \text{Sin}[$\gamma$ ]) ;\\
\text{f6}=(1/(m V \text{Cos}[$\gamma$ ]))(\text{lift} \text{Sin}[$\phi$ ] + m \text{vwDot} \text{Cos}[$\psi$ ]);\\
\text{f7} = 0;\\
f=\{\text{f1},\text{f2},\text{f3},\text{f4},\text{f5},\text{f6},\text{f7}\};\\
\text{b1}=\{0,0, 0,0,0,0,1\};

\begin{doublespace}
\noindent\(\pmb{\text{}}\)
\end{doublespace}

\text{(*Find Lie brackets of length 2*)}\\
\text{partialf} = D[f,\{\text{stateX}\}];\\
\text{partialb1} = D[\text{b1}, \{\text{stateX}\}];\\
\text{(*}[f,\text{b1}]\text{*)}\\
\text{liebracketfb1} = \text{partialb1}.f - \text{partialf}.\text{b1};

\begin{doublespace}
\noindent\(\pmb{\text{}}\)
\end{doublespace}

\text{(*Find Lie brackets of length 3*)}\\
\text{partialLieBracketfb1}= D[\text{liebracketfb1},\{\text{stateX}\}];\\
\text{(*}[f,[f,\text{b1}]]\text{*)}\\
\text{liebracketfAndfb1} = \text{partialLieBracketfb1}.f-\text{partialf}.\text{liebracketfb1};\\
\text{(*}[\text{b1},[f,\text{b1}]\text{*)}\\
\text{liebracketb1Andfb1}= \text{partialLieBracketfb1}.\text{b1}-\text{partialb1}.\text{liebracketfb1};\\

\begin{doublespace}
\noindent\(\pmb{\text{}}\)
\end{doublespace}

\text{(*Find Lie brackets of length 4*)}\\
\text{partialLieBracketfAndfb1} = D[\text{liebracketfAndfb1},\{\text{stateX}\}];\\
\text{(*} [f,[f,[f,\text{b1}]]]\text{*)}\\
\text{liebracketfffb1} = \text{partialLieBracketfAndfb1}.f - \text{partialf}.\text{liebracketfAndfb1};\\

\begin{doublespace}
\noindent\(\pmb{\text{}}\)
\end{doublespace}

\text{(*Find Lie brackets of length 5*)}\\
\text{(*} [f,[f,[f,[f,\text{b1}]]]]\text{*)}\\
\text{partialliebracketfffb1} = D[\text{liebracketfffb1},\{\text{stateX}\}];\\
\text{liebracketffffb1} = \text{partialliebracketfffb1}.f-\text{partialf}.\text{liebracketfffb1};\\

\begin{doublespace}
\noindent\(\pmb{\text{}}\)
\end{doublespace}

\text{(*Find Lie brackets of length 6*)}\\
\text{(*} [f,[f,[f,[f,[f,\text{b1}]]]]]\text{*)}\\
\text{partialliebracketffffb1} = D[\text{liebracketffffb1},\{\text{stateX}\}];\\
\text{liebracketfffffb1} = \text{partialliebracketffffb1}.f-\text{partialf}.\text{liebracketffffb1};

\end{document}